\documentclass[english,twoside,12pt,fleqn]{article} 
\evensidemargin=\oddsidemargin   
\evensidemargin=\oddsidemargin
 \usepackage[dvips]{graphicx}
\usepackage[cp1251]{inputenc} 
\usepackage[english]{babel}  
\usepackage{extraipa}
\usepackage{amsmath,amsthm,amssymb,newlfont}
\usepackage{array,hhline}
\usepackage{cite}                     
\usepackage{mathabx,mathrsfs}          

\usepackage{color}
\usepackage[colorlinks,linkcolor=blue]{hyperref} 
\newcommand{\C}{\ensuremath {\mathcal{C}} } 
\renewcommand{\L}{\ensuremath {{\mathcal{L}(\mathcal{C},\mathbb R^d)}} } \newcommand{\R}  {\ensuremath {{\mathbb R^d}} }
\newcommand{\Rs} {\ensuremath {{\mathbb R^{(s+2)d}}} }
\newcommand{\Cs} {\ensuremath {{\overline{\mathcal{C}}}} }
\newcommand{\Omegas}[1]{\ensuremath {{\overline{\Omega}}_{#1}} }
\newcommand{\ZeroMatrix}{\ensuremath {\text{\bf {O}}} }
\newcommand{\norma}[1]{\bigl\|      #1 \bigr\|} 
\newcommand{\norm} [1]{\bigl\vvvert #1 \bigr\vvvert}  

\newcommand{\MyOverline}[1]{\overline{#1 \vphantom{\widetilde{a}}}}
\newcommand{\omegah}{\Bigl(\dfrac{\omega}{h}\Bigr)}

\newtheorem{MyTheorem}{Theorem}[section]
\newtheorem{MyCorollary}[MyTheorem]{Corollary}
\newtheorem{MyLemma}[MyTheorem]{Lemma}

\newtheorem{MyRemark}[MyTheorem]{Remark}
\theoremstyle{definition}
\newtheorem{MyDefinition}[MyTheorem]{Definition}

\newenvironment{emph_}{\bfseries}{}              

\usepackage{caption}
\newcounter{MyTable}[section]

\DeclareCaptionLabelFormat{MyLabelFormat}{\sc Table \thesection.#2.\ \ } \DeclareCaptionFormat{MyFormat}{#1#3} \DeclareCaptionLabelSeparator{MySeparator}{.\\}
\newenvironment*{MyTable}[1]              
{                                        
  \begin{table}[!h]
    \tabcolsep=1 mm
    \stepcounter{MyTable}
    \caption{#1}
  \end{table}
}

\title{Two-step Runge-Kutta methods for \\ Retarded Functional Differential Equations}
\date{}
\author{\\
A. Tuzov      \footnote{Department of Control systems, Siberian State Aerospace University, Krasnoyarsk,   Russia, e-mail: \mbox{e-mail: tuzov@sibsau.ru}},
}
\begin{document}
\abovedisplayskip=6pt \belowdisplayskip=6pt

\begin{center} \bf \LARGE
  Two-step Runge-Kutta methods for\\
  \vspace{0pt}  Retarded Functional Differential Equations
\end{center}
\vspace{-12pt}
\begin{center} \large
  \renewcommand{\thefootnote}{\fnsymbol{footnote}}
   Anton Tuzov      \footnote{Department of Control systems, Siberian State Aerospace University, Krasnoyarsk, Russia, \mbox{e-mail: tuzov@sibsau.ru} }
\end{center}
\vspace{-12pt}

\begin{abstract}
\noindent
  This paper presents a class of two-step Runge-Kutta (TSRK) methods for the numerical solution of Retarded Functional Differential Equations (RFDEs).
  A convergence theorem is formulated and proved. Explicit methods up to order five are constructed. To avoid order reduction for mildly stiff problems
  the uniform stage order of the methods is chosen to be closed to uniform order.


\end{abstract}
\vspace{-24pt}
\section{Two-step Runge-Kutta methods for \\Ordinary Differential Equations}
\vspace{-6pt} For the numerical approximation of the solution $y(t)$ of a system of Ordinary Differential Equations
\begin{equation}
  \begin{split}
    y\,'(t) &=f(t,y), \quad t \in [t_0,T],\\
    y(t_0)&=y_0,
  \end{split}
\end{equation}
where $f:\ \mathbb R \times \R \rightarrow \R, \quad y_0 \in \R$,\\
we consider the class of General Linear Methods\cite{Butcher,Jackiewicz}
\begin{alignat}{2}\label{GLMs}
  Y^{[n]}_i&=\sum\limits_{j=1}^s a_{ij}\, hF^{[n]}_j+\sum\limits_{j=1}^r u_{ij} y^{[n-1]}_j, & \qquad &i=1,\dots ,s,  \notag\\
  y^{[n]}_i&=\sum\limits_{j=1}^s b_{ij}\, hF^{[n]}_j+\sum\limits_{j=1}^r v_{ij} y^{[n-1]}_j, &        &i=1,\dots ,r,  \\
  F^{[n]}_i&=f(t_{n-1}+c_i h,Y^{[n]}_i),                                                     &        &i=1,\dots ,s,  \notag
\end{alignat}
where $y^{[n-1]}_1,\dots,y^{[n-1]}_r$           --- input vectors, available at step number $n$,\linebreak[3]
      $Y^{[n]}_1,\dots,Y^{[n]}_s$               --- stage values,\linebreak[3]
      $F^{[n]}_1,\dots,F^{[n]}_s$               --- derivative values,\linebreak[3]
      $a_{ij},\ u_{ij},\ b_{ij},\ v_{ij},\ c_i$ --- coefficients of the method,
      $h$~\nobreakdash                          --- integration stepsize.

\enlargethispage{2\baselineskip}
Let us restrict ourselves to two-step methods and choose $r=s+2$ \mathindent=0em
\begin{alignat*}{3}
  y^{[n-1]}_1&\approx y(t_{n-1}), \quad & y^{[n-1]}_2&\approx y(t_{n-2}), \quad & y^{[n-1]}_{2+i}&\approx h y\,'(t_{n-2}+c_i h),\qquad i=1,\dots ,s.\quad \text{Then}\\
  y^{[n-1]}_1&=y_{n-1},                 & y^{[n-1]}_2&=y_{n-2},                 & y^{[n-1]}_{2+i}&=h f(t_{n-2}+c_i h,Y^{[n-1]}_i)=hF^{[n-1]}_i,\quad i=1,\dots,s,
\end{alignat*}
\mathindent=2.5em and \eqref{GLMs} takes the form
\begin{alignat}{2}\label{two_step_GLMs_draft_ODEs}
  Y^{[n]}_i&=h \sum\limits_{j=1}^s a_{ij}\, F^{[n]}_j+\ u_{i1}y_{n-1}+u_{i2}y_{n-2}+h\sum\limits_{j=1}^s u_{i,\,2+j} F^{[n-1]}_j, & \qquad &i=1,\dots ,s,  \notag\\
  y_n      &=h \sum\limits_{j=1}^s b_{1j}\, F^{[n]}_j+\ v_{11}y_{n-1}+v_{12}y_{n-2}+h\sum\limits_{j=1}^s v_{1,\,2+j} F^{[n-1]}_j,                          \\
  F^{[n]}_i&=f(t_{n-1}+c_i h,Y^{[n]}_i),                                                                                          &        &i=1,\dots ,s.  \notag
\end{alignat}
In the construction of GLMs it is assumed that $y^{[n-1]}_i=u_i\, y(t_{n-1})+v_i\, hy\,'(t_{n-1})+O(h^2)$ and 'preconsistensy conditions' holds
\begin{equation}\label{preconsistency_conditions}
  \begin{split}
    Vu=u,\\
    Uu={\mathbf 1}.
  \end{split}
\end{equation}
For \eqref{two_step_GLMs_draft_ODEs} we have $u_1=1,\ v_1=0,\quad u_2=1,\ v_2=-1,\quad u_{2+i}=0,\ v_{2+i}=1,\; i=1,\dots,s$. \linebreak[3] It follows from \eqref{preconsistency_conditions} that
\begin{align*}
  u_{i2}&=1-u_{i1},\qquad i=1,\dots,s,\\
  v_{12}&=1-v_{11}.
\end{align*}
Let us denote
\begin{alignat*}{5}
  K^{[n]}_i&:=F^{[n]}_i  &\quad \widetilde{a}_{ij}&:=u_{i,\,2+j}, & \quad u_i&:=u_{i1}, &\quad \Rightarrow u_{i2}&=1-u_i, &\qquad j&=1,\dots,s,\ i=1,\dots,s,\\
        b_j&:=b_{1j},    &         \widetilde{b_j}&:=v_{1,\,2+j}, &         v&:=v_{11}, &      \Rightarrow v_{12}&=1-v,   &       j&=1,\dots,s,
\end{alignat*}
then the method \eqref{two_step_GLMs_draft_ODEs} satisfying 'preconsistensy conditions'~\eqref{preconsistency_conditions} takes the form
\begin{alignat}{4}\label{two_step_GLMs_ODEs}
  y_n      &=  \ (1-v) y_{n-2}+v y_{n-1}   & &+h\sum\limits_{j=1}^s \widetilde{b_j}   K^{[n-1]}_j & &+h \sum\limits_{j=1}^s b_j   K^{[n]}_j, & \qquad                \notag\\
  K^{[n]}_i&=f(t_{n-1}+c_i h,Y^{[n]}_i),   & &                                                    & &                                        &       &i=1,\dots ,s, \\
  Y^{[n]}_i&=  (1-u_i) y_{n-2}+u_i y_{n-1} & &+h\sum\limits_{j=1}^s \widetilde{a}_{ij}K^{[n-1]}_j & &+h\sum\limits_{j=1}^s a_{ij} K^{[n]}_j, &       &i=1,\dots ,s. \notag
\end{alignat}

\section{Two-step Runge-Kutta methods for \\Retarded Functional Differential Equations}
We begin with notations introduced in~\cite{Maset} (see also~\cite{Maset_Zennaro}).\\
Let $r\in[0,+\infty)$, and \C be the space of continuous functions $[-r,0]\rightarrow \R$, equipped with the uniform norm $\norma{\varphi}_\C=\max\limits_{\theta \in [-r,0]}\norma{\varphi(\theta)}_\R
\ , \quad \varphi \in\C$, \; where $\norma{\cdot}_\R$ is an arbitrary norm on \R.\linebreak[3] Let $u$ be a continuous function $[a-r,b)\rightarrow \R$, where $a<b$. Then $\forall \, t\in[a,b)$ shift
function is defined by $u_t(\theta)=u(t+\theta), \quad \theta\in[-r,0]$ and $u_t\in\C$.

Let $s$ be stage number of the method, \linebreak[3] $\| \cdot \|$ be the norm on \Rs defined as the maximum of the norms of $s+2$ subvectors in \R: \mathindent=0em
\begin{align*}
  \Bigl( \forall \; z=(z_1,\dots ,z_{s+2})^T, \quad z_j \in \R,\; j=1,\dots,s+2 \Bigr):                         & \quad \norma{z}=\max\limits_{j=1,\dots,s+2} \norma{z_j}_\R ,
\end{align*}
and \Cs be the space of continuous functions  $[-r,0]\rightarrow \Rs$, equipped with the uniform norm:
\begin{align*}
  \Bigl( \forall \; \varphi=(\varphi_1,\dots ,\varphi_{s+2})^T, \quad \varphi_j \in \C,\; j=1,\dots,s+2 \Bigr): & \quad \norm{\varphi}=\max\limits_{\theta \in [-r,0]} \norma{\varphi(\theta)}=\max\limits_{j=1,\dots,s+2} \norma{\varphi_j}_\C .
\end{align*}
\mathindent=2.5em

\noindent Consider a system of Retarded Functional Differential Equations (RFDEs)
\begin{align}\label{RFDEs_draft}
  y\,'(t) =f(t,y_t),
\end{align}
where $f:\ \Omega \rightarrow \R,\quad \Omega = \mathbb R \times \C$.

It is assumed that the function $f$ is continuous and has the first derivative with respect to the second argument $f\,': \quad \Omega \rightarrow \L$, which is bounded and continuous with respect to the second argument. Then for each $(\sigma,\varphi)\in\Omega$ there exists an unique non-continuable solution $y=y(\sigma,\varphi): \quad [\sigma-r,\overline{t}\,) \rightarrow \R$
of~\eqref{RFDEs_draft} through $(\sigma,\varphi)$, where $\overline{t}=\overline{t}(\sigma,\varphi)$, $\overline{t} \in (\sigma,\infty]$, i.e. $y(t)$ satisfies~\eqref{RFDEs_draft} for $t\in
[\sigma,\overline{t}\,)$ and $y_{\sigma}(\theta)=\varphi(\theta)$ for $\theta \in [-r,0]$.

We next consider the system~\eqref{RFDEs_draft} through $(t_0,\phi_{\,0})\in\Omega$ on the $[t_0,T]$, where $T\in(t_0,\overline{t}(t_0,\phi_{\,0}))$ \vspace{-\baselineskip}
\begin{equation}\label{RFDEs}
  \begin{split}
     y\,'(t)         &=f(t,y_t),           \quad      t      \in [t_0,T],\\
     y_{t_0}(\theta) &=\phi_{\,0}(\theta), \quad \;\; \theta \in [-r,0].
  \end{split}
\end{equation}

The first derivative of $f$ with respect to the second argument is bounded, hence $f$ satisfies the Lipschitz condition with respect to the second argument
\begin{align}\label{Lipschitz_condition}
  \Bigl( \forall \; (t,\varphi_{\,t}),(t,\psi_{\,t})\in \Omega  \Bigr): \quad \norma{f(t,\varphi_{\,t})-f(t,\psi_{\,t})}_\R \le L \norma{\varphi_{\,t}-\psi_{\,t}}_\C,
\end{align}
where $L=\sup\limits_{(t,\, \varphi_{t}) \, \in \, \Omega}\norma{f\,'(t,\varphi_{\,t})}_\L$.

We introduce the class of two-step Runge-Kutta (TSRK) methods for RFDEs on the basis of approach proposed in~\cite{Maset_Zennaro,Maset}. We can reformulate the method~\eqref{two_step_GLMs_ODEs} for RFDEs~\eqref{RFDEs} with
the \begin{emph_}equispaced\end{emph_} mesh \linebreak[3] $\Delta_N=\left\{t_n:\quad t_n=t_0+nh,\; n=0,\dots,N  \right\}$, $\quad h=\dfrac{T-t_0}{N} \;$ as follows

\medmuskip=2mu \mathindent=0em
\smallskip
\begin{subequations}\label{TS_GLMs}
  \begin{alignat}{4}
    \eta^{[n]}(\alpha h) &=(1-v(\alpha))\eta^{[n-1]}(0)                          & &+v(\alpha)\eta^{[n-1]}(h)   & &+h\sum\limits_{j=1}^s \widetilde{b_j}(\alpha)K^{[n-1]}_j    & &+h\sum\limits_{j=1}^s b_j(\alpha)K^{[n]}_j,                 \label{TS_GLMs:eta}\\
                         &                                                       & &                            & &\ \alpha\in[0,1],                                                                                                          \notag\\
    K^{[n]}_i            &=f(t_{n-1}+c_i h,{Y_{\ c_ih}^{[n]\; i}}),              & &                            & &                                                            & &\ i=1,\dots ,s,                                             \label{TS_GLMs:Ki}\\
    Y^{[n]\; i}(\alpha h)&=(1-u_i(\alpha))\eta^{[n-1]}(0)                        & &+u_i(\alpha)\eta^{[n-1]}(h) & &+h\sum\limits_{j=1}^s \widetilde{a}_{ij}(\alpha)K^{[n-1]}_j & &+h\sum\limits_{j=1}^s a_{ij}(\alpha)K^{[n]}_j,              \label{TS_GLMs:Yi}\\
                         &                                                       & &                            & &\ \alpha\in [0,c_i],                                        & &\ i=1,\dots,s,                                              \notag\\
    \eta^{[n]}(\theta)   &=\eta^{[n-1]}_{\: h}(\theta),                          & &                            & &\ \theta \in [-r,0],                                                                                                       \label{TS_GLMs:eta_extended}\\
    Y^{[n]\; i}(\theta)  &=\eta^{[n-1]}_{\: h}(\theta),                          & &                            & &\ \theta \in [-r,0],                                        & &\ i=1,\dots ,s,                                             \label{TS_GLMs:Yi_extended}
  \end{alignat}
\end{subequations}
\mathindent=2.5em
\begin{alignat*}{2}
  \intertext{where }
  &\eta^{[n-1]}:\ [-r,h]\rightarrow \R, \quad K^{[n-1]}_i \in \R                                                 &\quad        &\text{are available as approximations}\\
  &                                                                                                              &             &\text{computed in the step } n-1,\\
  &Y^{[n]\; i}                                                                                                   &\text{--- }  &\text{stage functions},\\
  &K^{[n]}_i                                                                                                     &\text{--- }  &\text{stage values},\\
  &u_i(\cdot),\ \widetilde{a}_{ij}(\cdot), a_{ij}(\cdot),\quad v(\cdot),\ \widetilde{b_j}(\cdot),\ b_j(\cdot) \  &\text{--- }  &\text{coefficients of the method}.
\end{alignat*}
\begin{center}\label{MyPic_1}
  \includegraphics[scale=0.75]{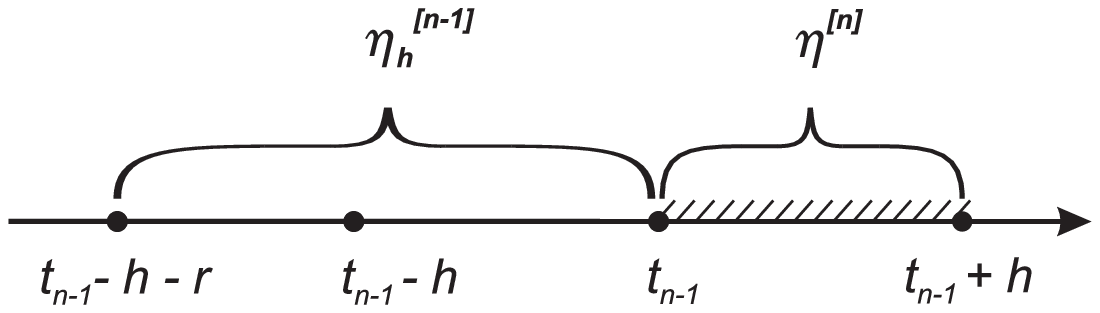}
\end{center}
\medmuskip=4mu Thus $\eta^{[n]}(\alpha h) \approx y(t_{n-1}+\alpha h), \; \alpha \in [0,1]$, $\quad \eta^{[n]}(\theta) \approx y(t_{n-1}+\theta),\; \theta \in [-r,0]$, hence \linebreak[3]
$\eta^{[n]}_{\: h}\approx y_{t_n}$ on $[-r-h,0]$, where $t_n=t_{n-1}+h$.

\pagebreak[4]
It is assumed that coefficients $(a_{ij}(\cdot),b_j(\cdot),c_i,\widetilde{a}_{ij}(\cdot),\widetilde{b}_j(\cdot),u_i(\cdot),v(\cdot))_{\,i,\, j=1,\dots,s}\ $ of TSRK methods satisfy the following
conditions: \mathindent=1.5em
\begin{subequations}\label{coeff_basic assumptions}
\begin{alignat}{4}
  \text{\textbullet}\ & a_{ij}(\cdot),\ \widetilde{a}_{ij}(\cdot),\ u_i(\cdot), &\quad &\text{are polynomial functions } [0,c_i]\rightarrow {\mathbb R}, & \qquad i,\ j&=1,\dots,s.\\
                      & b_j(\cdot),\    \widetilde{b}_j(\cdot),\      v(\cdot), &      &\text{are polynomial functions }   [0,1]\rightarrow {\mathbb R}, &            j&=1,\dots,s.\notag\\
  \text{\textbullet}\ & c_i \in {\mathbb R},\ c_i \ge 0,                        &      &                                                                 &            i&=1,\dots,s.\\
  \text{\textbullet}\ & a_{ij}(0)=\widetilde{a}_{ij}(0)=0,                      &      & u_i(0)=1,                                                       &        i,\ j&=1,\dots,s.  \label{Yi_continuous}\\
  \text{\textbullet}\ & b_j(0)\ =\widetilde{b}_j(0)\ =0,                        &      & v(0)=1,                                                         &            j&=1,\dots,s.  \label{eta_continuous}
\end{alignat}
\end{subequations}
\mathindent=2.5em

\noindent
The last two conditions correspondingly 
guarantee continuity of the stage functions $Y^{[n]\; i}_{\ c_i h} \in \C$ and the approximate solution $\eta^{[n]}_{\: h} \in \C$ provided that approximate solution computed in the previous step is
continuous function $\eta^{[n-1]}_{\: h} \in \C.$

\begin{MyRemark}
  If the conditions
  \begin{alignat}{4}\label{one_step_conditions}
    u_i(\cdot)&=1, & \qquad \widetilde{a}_{ij}(\cdot)&=0, &\qquad i,\ j&=1,\dots,s,\\
      v(\cdot)&=1, &       \widetilde{b}_j(\cdot)   &=0,  &           j&=1,\dots,s,\notag
  \end{alignat}
  hold, the two-step method~\eqref{TS_GLMs} becomes the one-step RK method for RFDEs  introduced in~\cite{Maset_Zennaro,Maset}, where initial function $\ \varphi:=\eta^{[n-1]}_{\: h}$.
\end{MyRemark}
\begin{MyDefinition}
  The TSRK method with coefficients $(a_{ij}(\cdot),b_j(\cdot),c_i,\widetilde{a}_{ij}(\cdot),\widetilde{b}_j(\cdot),u_i(\cdot),v(\cdot))_{i,j=1,\dots,s}\ $
  is called explicit if $a_{ij}(\cdot)=0 \text{ for all $j$ such that} j\ge i,\quad i,j=1,\dots,s.$
\end{MyDefinition}

Make the change of the independent variables: $\alpha h=\omega+c_i h,\; \theta=\omega+c_i h$ and introduce the shifted coefficient functions:
\begin{align}\label{shifted_coeff}
  \MyOverline{u}_i(\cdot)=u_i(\,\cdot+c_i), \quad \MyOverline{\widetilde{a}}_{ij}(\cdot)=\widetilde{a}_{ij}(\,\cdot+c_i), \quad \MyOverline{a}_{ij}(\cdot)={a}_{ij}(\,\cdot+c_i), \qquad i,\ j=1,\dots,s.
\end{align}
Then~\eqref{TS_GLMs:Yi} can be rewritten in the form: \mathindent=0em
\begin{alignat}{3}\label{TS_GLMs_shifted_Yi}
  Y^{[n]\; i}_{\ c_i h}(\omega)=\,&\left(1-\MyOverline{u}_i\omegah \right)  \eta^{[n-1]}(0)                                                            & +\, &\MyOverline{u}_i\omegah \eta^{[n-1]}(h)+ \notag\\
                               +\,&h\sum\limits_{j=1}^s \MyOverline{\widetilde{a}}_{ij}\omegah f \left(t_{n-2}+c_j h,\ {Y_{\ c_jh}^{[n-1]\; j}}\right) & +\, &h\sum\limits_{j=1}^s \MyOverline{a}_{ij}\omegah f \left(t_{n-1}+c_j h,\ {Y_{\ c_jh}^{[n]\; j}}\right),\\
                                  &\omega\in[-c_ih,0],                                                                                                 &  \, &i=1,\dots,s.                     \notag
\end{alignat}
\mathindent=2.5em Let us introduce the following notation: \abovedisplayskip=2pt \belowdisplayskip=2pt
\begin{align}\label{coeff_notations}
  c_{s+1}:=1,\quad u_{s+1}(\cdot):=v(\cdot),\quad \widetilde{a}_{s+1,\,j}(\cdot):=\widetilde{b}_j(\cdot),\quad a_{s+1,\,j}(\cdot):=b_j(\cdot),\qquad j=1,\dots,s.
\end{align}
\abovedisplayskip=6pt \belowdisplayskip=6pt This notation allows us to rewrite~\eqref{TS_GLMs:eta} like~\eqref{TS_GLMs_shifted_Yi}: \mathindent=0em
\begin{alignat}{2}\label{TS_GLMs_shifted_eta}
  \eta^{[n]}_{\: h}(\omega) =&   \left(1-\MyOverline{u}_{s+1}\omegah \right)  \eta^{[n-1]}(0)                                                                & +&\MyOverline{u}_{s+1}\omegah \eta^{[n-1]}(h)+ \notag\\
                            +&\, h\sum\limits_{j=1}^s \MyOverline{\widetilde{a}}_{{s+1},\,j}\omegah f \left(t_{n-2}+c_j h,\ {Y_{\ c_jh}^{[n-1]\; j}}\right)  & +&h\sum\limits_{j=1}^s \MyOverline{a}_{{s+1},\,j}\omegah f \left(t_{n-1}+c_j h,\ {Y_{\ c_jh}^{[n]\; j}}\right), \hspace{-5cm}\\
                             &\, \omega\in[-h,0].\notag
\end{alignat}
Denote $z^{[n]}(\omega)=\left( Y^{[n]\; 1}_{\ c_1 h}(\omega),\dots,Y^{[n]\; s}_{\ c_s h}(\omega), \eta^{[n]}_{\: h}(\omega), \eta^{[n]}(0) \right)^T,\qquad n=1,2,\dots \; .$ \\
Then the method~\eqref{TS_GLMs} can be reformulated as follows:

\mathindent=0em
\begin{alignat}{2}\label{TS_GLMs_shifted}
  z^{[n]\; i}(\omega)=  & \left(1-\MyOverline{u}_i\omegah \right) z^{[n-1],\; {s+2}}                                                   & +\,& \MyOverline{u}_i\omegah z^{[n-1],\; {s+1}}(0)+                                                  \notag\\
                     +\,& h\sum\limits_{j=1}^s \MyOverline{\widetilde{a}}_{ij}\omegah f \left(t_{n-2}+c_j h,\ {z^{[n-1]\; j}\,}\right) & +\,& h\sum\limits_{j=1}^s \MyOverline{a}_{ij}\omegah f \left(t_{n-1}+c_j h,\ {z^{[n]\; j}\,}\right), \notag\\
                      \,& \omega\in[-c_ih,0],                                                                                          &  \,& i=1,\dots,s+1,                                                                          \\[6pt]
  z^{[n]\; i}(\omega)=\,& z^{[n-1],\; {s+1}}_{\, c_i h}(\omega),                                                                                                                                                                           \notag\\
                      \,& \omega\in[-r,-c_ih],                                                                                         &  \,& i=1,\dots ,s+1,                                                                           \notag\\[6pt]
  z^{[n],\; s+2}     =\,& z^{[n-1],\; {s+1}}(0).                                                                                                                                                                                           \notag
\end{alignat}
\mathindent=2.5em

Denote $\Omegas{h}=\Bigl\{(\sigma,\varphi)\in \mathbb R \times \Cs\;  \Big| \bigl(\sigma+(c_j-1)h,\varphi_j\bigr)\in \Omega,  \quad j=1,\dots,s+1  \Bigr\}$.

When \begin{emph_}one\end{emph_} step of the method~\eqref{TS_GLMs_shifted} is applied with stepsize $h$ to~\eqref{RFDEs} for the computation of the solution through
$(t_{n-1},z^{[n-1]})\in\Omegas{h}$ on $[t_{n-1},t_n]$, it yields a continuous function \linebreak[4] $z^{[n]}=z^{[n]}(h,t_{n-1},z^{[n-1]})$ in \Cs, where $z^{[n]\; i}     \approx y_{t_{n-1}+c_ih}, \;
i=1,\dots,{s+1},$ $z^{[n],\; {s+2}} \approx y(t_{n-1})$. \linebreak[3] Note that $z^{[n],\; {s+1}}=\eta^{[n]}_h \approx y_{t_{n}}$, since $c_{s+1}=1$. The approximation $z^{[n]}$ is defined for
stepsizes $h\in (0,\overline{h}\,)$, where $\overline{h}=\overline{h}(t_{n-1},z^{[n-1]})$, moreover $(t_{n},z^{[n]})\in\Omegas{h}$.

Let $z^{[0]}$ be computed by some starting method $S:\; z^{[0]}=S(h,t_0,\phi_{\,0})$. When \begin{emph_}several\end{emph_} steps of the method~\eqref{TS_GLMs_shifted} are applied to~\eqref{RFDEs} for
the computation of the solution through $(t_{0},z^{[0]})\in\Omegas{h}$ on $[t_0,T]$, they yield the finite sequence of continuous functions $\{ z^{[n]} \}_{n=1,\dots,N\;}$ in \Cs. These sequence is
defined for stepsizes $h\in (0,\overline{h}_N\,)$, where
$\overline{h}_N=\min\limits_{n=1,\dots,N\;} \overline{h}(t_{n-1},z^{[n-1]})$, moreover $(t_{n},z^{[n]})\in\Omegas{h}$ for all $n=1,\dots,N\;$.

\enlargethispage{0.5\baselineskip}
Rewrite the method~\eqref{TS_GLMs_shifted} in matrix-vector form. For this purpose we define \mathindent=0em \thickmuskip=4mu
\begin{alignat*}{2}
  &\widetilde{\mathcal{A}}(\omega)=
  \begin{pmatrix}
    \MyOverline{\widetilde{a}}_{11}\omegah      & \cdots  & \MyOverline{\widetilde{a}}_{1 s}\omegah     & 0      & 0      \\[-6pt]
    \vdots                                      &         & \vdots                                      & \vdots & \vdots \\[-9pt]
    \MyOverline{\widetilde{a}}_{s+1,\,1}\omegah & \cdots  & \MyOverline{\widetilde{a}}_{s+1,\,s}\omegah & 0      & 0      \\[-9pt]
    0                                           & \cdots  & 0                                           & 0      & 0      \\
  \end{pmatrix}, \;
  \mathcal{A}(\omega)=
  \begin{pmatrix}
    \MyOverline{a}_{11}\omegah      & \cdots  & \MyOverline{a}_{1 s}\omegah     & 0      & 0      \\[-6pt]
    \vdots                          &         & \vdots                          & \vdots & \vdots \\[-9pt]
    \MyOverline{a}_{s+1,\,1}\omegah & \cdots  & \MyOverline{a}_{s+1,\,s}\omegah & 0      & 0      \\[-9pt]
    0                                         & \cdots  & 0                     & 0      & 0      \\
  \end{pmatrix}, \\
  &M(\omega)=
  \begin{pmatrix}
    \ZeroMatrix   & \MyOverline{u}\,\omegah & e-\MyOverline{u}\,\omegah  \\
    {\mathbf 0}^T & 1                       & 0                          \\
  \end{pmatrix},\\
  &\varPhi(t_{n},z^{[n]},h) =\Bigl(f \bigl(t_{n}+(c_1-1)h,z^{[n]\; 1} \bigr),\dots,f \bigl(t_{n}+(c_s-1)h,z^{[n]\; s} \bigr),0,0 \Bigr)^T, &
\end{alignat*}
\mathindent=2.5em \thickmuskip=5mu
  where $e=\Bigl(1,\dots,1\Bigr)^T,\quad
    \MyOverline{u}\,\omegah=\Bigl(\MyOverline{u}_1\omegah,\dots,\MyOverline{u}_{s+1}\omegah\Bigr)^T,\quad$\\
    \ZeroMatrix is $(s+1)\times s$~\nobreakdash-~dimensional zero matrix.

Then the method~\eqref{TS_GLMs_shifted} can be represented in the form of $A$\nobreakdash-method~\cite{Albrecht_1,Albrecht_2}:
\begin{subequations}\label{TS_GLMs_matrix_form}
  \begin{alignat}{3}
    z^{[n]}(\omega)    &=M(\omega)z^{[n-1]}(0)                  & +\,& h\, \widetilde{\mathcal{A}}(\omega)\,\varPhi(t_{n-1},z^{[n-1]},h) & +\,& h\, \mathcal{A}(\omega)\,\varPhi(t_{n},z^{[n]},h), \label{TS_GLMs_matrix_form:z} \\
    z^{[n]\; i}(\omega)&=z^{[n-1],\; {s+1}}_{\, c_i h}(\omega), &  \,& \omega\in[-r,-c_ih],                                              &  \,& i=1,\dots ,s+1,                                    \label{TS_GLMs_matrix_form:z_extended}
  \end{alignat}
\end{subequations}
\mathindent=0em
\vspace{-1.5\baselineskip}
\begin{alignat*}{4}
  &\text{where } & z^{[n]\; i}(\omega)     \text{ is computed by}&~\eqref{TS_GLMs_matrix_form:z},                  &\quad &\text{if } \omega\in[-c_ih,0],  &\quad &               \\
  &              &                                      \text{by}&~\eqref{TS_GLMs_matrix_form:z_extended},         &      &\text{if } \omega\in[-r,-c_ih], &      & i=1,\dots ,s+1;\\
  &              & z^{[n],\; {s+2}}(\omega) \text{ is constant and computed by} &~\eqref{TS_GLMs_matrix_form:z} &      &\text{for all } \omega.
\end{alignat*}
\noindent Denote the exact value function of the method~\eqref{TS_GLMs_matrix_form} by $Z^{[n]}(\omega)$ and the global error function by $q^{[n]}(\omega):$
\begin{alignat*}{3}
  Z^{[n]}(\omega)&=\left( y_{\,t_{n}+(c_1-1)h}(\omega),\dots,y_{\,t_{n}+(c_s-1)h}(\omega), y_{\,t_{n}}(\omega),  y_{\,t_{n-1}}(0) \right)^T, & \quad &\omega\in[-r,0], & \quad &n=1,2,\dots \;, \\
  q^{[n]}(\omega)&=Z^{[n]}(\omega)-z^{[n]}(\omega),                                                                                          &       &\omega\in[-r,0], &       &n=1,2,\dots \;.
\end{alignat*}
\begin{MyDefinition} \label{def_convergence_matrix_form}
  The method~\eqref{TS_GLMs_matrix_form} has uniform order of convergence~$q$ if
  \mathindent=2.5em
  \begin{align*}
    \max\limits_{n=1,\dots,N} \norm{q^{[n]}}=O(h^q)     \qquad \text{ as } h \rightarrow 0
  \end{align*}
  and discrete order of convergence~$p$ if
  \begin{align*}
    \max\limits_{n=1,\dots,N} \norma{q^{[n]}(0)}=O(h^p) \quad\, \text{ as } h \rightarrow 0,
  \end{align*}
  where  $q$ and $p$ are positive integers.
\end{MyDefinition}

\noindent Denote by $hd^{[n]}$ the local discretization error function obtained as the residual function by replacing $z^{[n-1]},\; z^{[n]}$ in~\eqref{TS_GLMs_matrix_form} by the exact value
functions $Z^{[n-1]},\; Z^{[n]}$ correspondingly:
\begin{subequations}\label{TS_GLMs_matrix_form_residual}
  \begin{alignat}{2}
    Z^{[n]}(\omega)    &=M(\omega)Z^{[n-1]}(0)                 & &+h\, \widetilde{\mathcal{A}}(\omega)\,\varPhi(t_{n-1},Z^{[n-1]},h) +h\, \mathcal{A}(\omega)\,\varPhi(t_{n},Z^{[n]},h)   +hd^{[n]}(\omega),     \hspace{-5cm}    \label{TS_GLMs_matrix_form_residual:Z} \\
    Z^{[n]\; i}(\omega)&=Z^{[n-1],\; {s+1}}_{\, c_i h}(\omega) & &+hd^{[n]\; i}(\omega),                                              \qquad \omega\in[-r,-c_ih],                          \quad i=1,\dots ,s+1.                  \label{TS_GLMs_matrix_form_residual:Z_extended}
  \end{alignat}
\end{subequations}
\begin{MyDefinition}\label{def_consistency_matrix_form}
  \mathindent=2.5em
  The method~\eqref{TS_GLMs_matrix_form} has uniform order of consistency~$q$ if
  \begin{align*}
    \max\limits_{n=1,\dots,N} \norm{d^{[n]}}=O(h^q) \qquad \text{ as } h \rightarrow 0
  \end{align*}
  and discrete order of consistency~$p$ if
  \begin{align*}
    \max\limits_{n=1,\dots,N-1} \norma{d^{[n]}(0)}=O(h^p) \quad\, \text{ as } h \rightarrow 0.
  \end{align*}
\end{MyDefinition}

\noindent It follows from~\eqref{TS_GLMs_matrix_form_residual:Z_extended} that $d^{[n]\; i}(\omega)=0$ for $\omega\in[-r,-c_ih]$. Subtracting~\eqref{TS_GLMs_matrix_form}
from~\eqref{TS_GLMs_matrix_form_residual} we have
\begin{subequations}\label{TS_GLMs_matrix_form_global_error}
  \begin{alignat}{2}
    q^{[n]}(\omega)    &=M(\omega)q^{[n-1]}(0)                  +h\, \widetilde{\mathcal{A}}(\omega)\,\Delta\varPhi^{[n-1]}  +h\, \mathcal{A}(\omega)\,\Delta\varPhi^{[n]}   +hd^{[n]}(\omega), \label{TS_GLMs_matrix_form_global_error:q} \\
    q^{[n]\; i}(\omega)&=q^{[n-1],\; {s+1}}_{\, c_i h}(\omega),  \qquad \omega\in[-r,-c_ih],  \quad i=1,\dots ,s+1,                                                                             \label{TS_GLMs_matrix_form_global_error:q_extended}
  \end{alignat}
\end{subequations}
where $\Delta\varPhi^{[n]}=\varPhi(t_{n},Z^{[n]},h)-\varPhi(t_{n},z^{[n]},h)$.

\noindent \thickmuskip=4mu We further extend the shifted coefficient functions $\MyOverline{u}_i\omegah, \; \MyOverline{\widetilde{a}}_{ij}\omegah,\; \MyOverline{a}_{ij}\omegah$ \linebreak[3] to
$\omega\in[-r,-c_ih)$~by
\begin{align*}
  \MyOverline{u}_i\omegah=1, \quad \MyOverline{\widetilde{a}}_{ij}\omegah=0,\quad \MyOverline{a}_{ij}\omegah=0, \qquad \omega\in[-r,-c_ih), \ i=1,\dots,s+1, \ j=1,\dots,s.
\end{align*}
This extension is continuous, since $\MyOverline{u}_i\omegah \Bigr|_{\omega=-c_ih}=u_i(0)=1, \quad \MyOverline{\widetilde{a}}_{ij}\omegah\Bigr|_{\omega=-c_ih}=\widetilde{a}_{ij}(0)=0, \quad$
$\MyOverline{a}_{ij}\omegah\Bigr|_{\omega=-c_ih}=a_{ij}(0)=0, \qquad i=1,\dots,s+1, \ j=1,\dots,s.$ Thus, the elements of $M(\cdot)$, $\widetilde{\mathcal{A}}(\cdot)$, $\mathcal{A}(\cdot)$ are
continuous functions $[-r,0]\rightarrow {\mathbb R}$. From~\eqref{TS_GLMs_matrix_form_global_error:q} we have \thickmuskip=5mu
\begin{alignat}{4}\label{global_error_estimate_0}
  \norma{q^{[n]}(\omega)} &\le \norma{M(\omega)}\norma{q^{[n-1]}(0)} & &+h\, \norma{\widetilde{\mathcal{A}}(\omega)} \norma{\Delta\varPhi^{[n-1]}} & &  +h\, \norma{\mathcal{A}(\omega)}\norma{\Delta\varPhi^{[n]}} & &+h\norma{d^{[n]}(\omega)} \le      \\
                          &\le \norm{M}\norma{q^{[n-1]}(0)}          & &+h\, \norm{\widetilde{\mathcal{A}}}          \norma{\Delta\varPhi^{[n-1]}} & &  +h\, \norm{\mathcal{A}}\norma{\Delta\varPhi^{[n]}}          & &+h\norm{d^{[n]}},\; n=1,2,\dots \notag
\end{alignat}
It follows from~\eqref{TS_GLMs_matrix_form_global_error:q} for $\omega =0$ that
\begin{alignat}{3}\label{global_error_estimate_2}
  q^{[n-1]}(0) &= M^{\,n-1}(0)q^{[0]}(0) & &+h\, \sum\limits_{l=1}^{n-1} M^{\, n-1-l}(0) \widetilde{\mathcal{A}}(0) \Delta\varPhi^{[l-1]} &+ &\\
               &                         & &+h\, \sum\limits_{l=1}^{n-1} M^{\, n-1-l}(0) \mathcal{A}(0) \Delta\varPhi^{[l]}               &+ &h\, \sum\limits_{l=1}^{n-1} M^{\, n-1-l}(0) d^{[l]}(0),\; n=2,3,\dots \hspace{-20cm} \notag
\end{alignat}

\begin{MyDefinition}
  The TSRK method~\eqref{TS_GLMs_matrix_form} is called zero-stable if $M(0)$ is power bounded, i.e.
  \begin{align}\label{power_boundedness}
    \bigl( \exists \, C>0 \bigr): \qquad \norma{M^{\, l}(0)} \le C, \quad  l=1,2,\dots \;.
  \end{align}
\end{MyDefinition}

\begin{MyLemma}
  The method~\eqref{TS_GLMs_matrix_form} is zero-stable iff \; $0 \le \MyOverline{u}_{s+1}(0)<2$.
\end{MyLemma}
\begin{proof}
The condition~\eqref{power_boundedness} holds iff~\cite{Butcher}
\begin{enumerate}
  \item\label{root_cond_1}
  the minimal polynomial of $M(0)$ has all its zeros in the closed unit disc and
  \item\label{root_cond_2}
  all its multiple zeros in the open unit disc.
\end{enumerate}

{ \linespread{2} \normalsize
Denote by $\psi(\lambda)$ and 
$\Delta(\lambda)$ the minimal and characteristic polynomials of $M(0)$ correspondingly. Then $\psi(\lambda)=\dfrac{\Delta(\lambda)}{D_{s+1}(\lambda)}\, , \;$ where $D_{s+1}(\lambda)$ is the greatest
common divisor of $(s+1)$~\nobreakdash-~t order minors of the characteristic matrix $\lambda I-M(0)$.

} \vspace{-1\baselineskip} \mathindent=2.5em
\begin{align*}
  \Delta(\lambda)=\lambda^s(\lambda-1)(\lambda-(v-1)), \quad D_{s+1}(\lambda)=\lambda^{s-1}, \quad \text{where } v=\MyOverline{u}_{s+1}(0).
\end{align*}
Hence $\psi(\lambda)=\lambda(\lambda-1)(\lambda-(v-1))$ and zeros of the minimal polynomial $\psi(\lambda)$ are
\begin{align*}
  \lambda_1=0, \quad \lambda_2=1, \quad \lambda_3=v-1.
\end{align*}
The condition~\eqref{root_cond_1} holds iff $0 \le v \le 2$;~\eqref{root_cond_2} holds iff $v \ne 2$. Thus,~\eqref{power_boundedness} holds iff $0 \le v < 2$.
\end{proof}

\begin{MyTheorem}
  \thinmuskip=0mu
  \thickmuskip=4mu
  Assume that the method~\eqref{TS_GLMs_matrix_form} is zero-stable and
  the starting procedure $S$ for it, which specifies the starting value $z^{[0]}=S(h,t_0,\phi_{\,0})$, such that
  $\norm{q^{[0]}}=O(h^{p-1})$, $\norma{q^{[0]}(0)}=O(h^{p})$ as $h \rightarrow 0$.

  If the method has uniform order of consistency $p-1: \; \max\limits_{n=1,\dots,N}\norm{d^{[n]}}=O(h^{p-1})$
  and discrete order of consistency $p: \; \max\limits_{n=1,\dots,N-1} \norma{d^{[n]}(0)}=O(h^p)$ as $h \rightarrow 0$,
  then it has uniform order of convergence  $p:$ \linebreak[3] $\max\limits_{n=1,\dots,N} \norm{q^{[n]}}=O(h^p), \quad h \rightarrow 0$.
\end{MyTheorem}
\begin{proof}
Since the function $f$ satisfies the Lipschitz condition~\eqref{Lipschitz_condition} with respect to the second argument with constant $L$, we have

\medmuskip=3mu \thickmuskip=4mu
\begin{align*}
  \norma{\Delta\Phi^{[n]}_i}_\R     &=\norma{\Phi_i(t_n,Z^{[n]},h)-\Phi_i(t_n,z^{[n]},h)}_\R=\norma{f(t_{n}+(c_i-1)h,Z^{[n] \; i})-\\
                                    &-f(t_{n}+(c_i-1)h,z^{[n] \; i})}_\R \le L \norma{Z^{[n]\; i}-z^{[n]\; i}}_\C=L \norma{q^{[n]\; i}}_\C \le L \norm{q^{[n]}}, \quad i=1,\dots,s,\\
  \norma{\Delta\Phi^{[n]}_{s+1}}_\R &=\norma{\Delta\Phi^{[n]}_{s+2}}_\R=0.
\end{align*}
\thinmuskip=3mu \medmuskip=4mu \thickmuskip=5mu Hence
\begin{align}\label{Delta_Phi_Lipschitz_condition}
  \norma{\Delta\Phi^{[n]}}\le L \norm{q^{[n]}}.
\end{align}

The method~\eqref{TS_GLMs_matrix_form} is zero-stable, i.e.~\eqref{power_boundedness} holds, where it can be assumed without loss of generality that $C \ge 1$. Denote $K=CL$. It follows
from~\eqref{global_error_estimate_2},~\eqref{power_boundedness} and~\eqref{Delta_Phi_Lipschitz_condition} that
\begin{alignat*}{4}
  \norma{q^{[n-1]}(0)} &\le    C \norma{q^{[0]}(0)} & &+h\, K \norma{\widetilde{\mathcal{A}}(0)} \sum\limits_{l=1}^{n-1} \norm{q^{[l-1]}} & +& \\
                       &                            & &+h\, K \norma{\mathcal{A}(0)}             \sum\limits_{l=1}^{n-1} \norm{q^{[l]}}   & +&\, h\,C \sum\limits_{l=1}^{n-1} \norma{d^{[l]}(0)}, \qquad n=1,2,\dots \; ,
\end{alignat*}
where we assume that the sum is zero if the lower summation index exceeds the upper one.
\linebreak[3]
By~\eqref{global_error_estimate_0} and~\eqref{Delta_Phi_Lipschitz_condition} it follows that
\begin{alignat}{2}\label{global_error_estimate_3}
  \norma{q^{[n]}(\omega)} &\le \norm{M}       \Bigl( C \norma{q^{[0]}(0)} + h\, K \norma{\widetilde{\mathcal{A}}(0)}  \norm{q^{[0]}}\Bigr) & +&h K \norm{M} \Bigl( \norma{\widetilde{\mathcal{A}}(0)}+\norma{\mathcal{A}(0)} \Bigr) \sum\limits_{l=1}^{n-2} \norm{q^{[l]}}+ \notag\\
                          &+h\,\Bigl( K \norm{M}\norma{\mathcal{A}(0)}+L\norm{\widetilde{\mathcal{A}}} \Bigr)\norm{q^{[n-1]}}              & +&h L \norm{\mathcal{A}}\norm{q^{[n]}}+\\
                          &+h\, C \norm{M}\sum\limits_{l=1}^{n-1} \norma{d^{[l]}(0)} +h\,\norm{d^{[n]}},                                   &  & \qquad\qquad\qquad\qquad\qquad\qquad n=1,2,\dots \notag
\end{alignat}

  Consider~\eqref{global_error_estimate_3} for $n=1,\dots,N$. Since $h\,\sum\limits_{l=1}^{n-1} \norma{d^{[l]}(0)} \le h\,(n-1) \max\limits_{l=1,\dots,N-1}\norma{d^{[l]}(0)}$
  \linebreak[3] and \linebreak[4] $h\,(n-1) \le h\, N = T-t_0$ we obtain $h
  \sum\limits_{l=1}^{n-1} \norma{d^{[l]}(0)}=O(h^p)$ as $h \rightarrow 0,\quad n=1,\dots,N$.

  Moreover,
  \mathindent=2.5em
  \begin{alignat*}{2}
    \norma{q^{[n]}(\omega)} &\le \widetilde{C}_0\, h^p+  D\,h\sum\limits_{l=1}^{n} \norm{q^{[l]}}, \qquad n=1,\dots,N \quad \text{for sufficiently small } h,
  \end{alignat*}
  where $\widetilde{C}_0>0,\; D>0$.

  By hypothesis, we have that $\norm{q^{[0]}} \le \widetilde{\widetilde{C}}_0 \, h^p$ for some $\widetilde{\widetilde{C}}_0>0$ and sufficiently small $h$.
  Denote ${C}_0=\max \Bigl\{ \widetilde{C}_0,\widetilde{\widetilde{C}}_0 \Bigr\}$. Then
  \begin{alignat}{2}\label{global_error_estimate_5}
    \norma{q^{[n]}(\omega)} &\le C_0\, h^p+  D\,h\sum\limits_{l=1}^{n} \norm{q^{[l]}}, \qquad n=0,\dots,N \quad \text{for sufficiently small } h.
  \end{alignat}
  Let $K_n$ be defined by
  \begin{align}\label{Kn}
    K_n=C_0 h^p+  D\,h\sum\limits_{l=1}^{n} K_l, \qquad n=0,1,2,\dots \;.
  \end{align}
  Using~\eqref{global_error_estimate_5},~\eqref{Kn} and~\eqref{TS_GLMs_matrix_form_global_error:q_extended}
  it is easy to prove by induction that
  \begin{align}
    \norm{q^{[n]}} \le K_n, \qquad  n=0,\dots,N \quad \text{for sufficiently small } h\in (0,\overline{h}\,),\; \overline{h} < D^{-1}.
  \end{align}

  Finally, it follows from~\eqref{Kn} that $K_n=(1-Dh)^{-1} K_{n-1}=\dots=(1-Dh)^{-n} K_0$, where
  \linebreak[4]
  $(1-Dh)^{-n} = \Bigl(1+\dfrac{Dh}{1-Dh} \Bigr)^n \le \exp \Bigl(\dfrac{D nh}{1-Dh}\Bigr) \le \exp \Bigl(\dfrac{D(T-t_0)}{1-D\overline{h}}\Bigr)$.
  Hence
  \begin{align}
    K_n \le C_0 \exp \Bigl(\dfrac{D(T-t_0)}{1-D\overline{h}}\Bigr) \, h^p, \qquad  n=0,\dots,N, \quad  h\in (0,\overline{h}\,),
  \end{align}
  which concludes the proof.
\end{proof}

\section{Order conditions}\label{sec_Order cond}
Assume that $f$ is of class $C^l$ with respect to the second argument for a sufficiently large $l$ and solution $y(t)$ of~\eqref{RFDEs} is of piecewise class $C^m$ for a sufficiently large $m$. Let
$c_1^*,\dots,c_{s^*}^*$ such that $c_1^*<c_2^*<\dots<c_{s^*}^*$ and $\{c_1^*,\dots,c_{s^*}^*\}=\{c_1,\dots,c_s \}$, i.e. $c_i^*$ are distinct $c_i$ in increasing order.

  Observe that convergence and consistency of the method~\eqref{TS_GLMs_matrix_form}
  in Definitions~\autoref{def_convergence_matrix_form} and~\autoref{def_consistency_matrix_form}
  means \begin{emph_}stage\end{emph_} convergence and consistency of the corresponding method~\eqref{TS_GLMs}.
  Now we consider weaker definitions.

\begin{MyDefinition}
  \mathindent=2.5em
  The method~\eqref{TS_GLMs} has uniform order of consistency~$q$ if
  \begin{align*}
    \max\limits_{n=1,\dots,N}   \norma{d^{[n]\; s+1}}_\C=O(h^q)      \qquad\qquad \text{ as } h \rightarrow 0
  \end{align*}
  and discrete order of consistency~$p$ if
  \begin{align*}
    \max\limits_{n=1,\dots,N-1} \norma{d^{[n]\; s+1}(0)}_\R=O(h^p) \,\quad\, \text{ as } h \rightarrow 0.
  \end{align*}
\end{MyDefinition}
\begin{MyDefinition}
  The method~\eqref{TS_GLMs} has uniform order of convergence~$q$ if
  for the corresponding method~\eqref{TS_GLMs_matrix_form} the following condition holds:
  \mathindent=2.5em
  \begin{align*}
    \max\limits_{n=1,\dots,N} \norma{q^{[n]\; s+1}}_\C=O(h^q)     \;\;\,\qquad \text{ as } h \rightarrow 0
  \end{align*}
  and discrete order of convergence~$p$ if
  \begin{align*}
    \max\limits_{n=1,\dots,N} \norma{q^{[n]\; s+1}(0)}_\R=O(h^p)  \quad \text{ as } h \rightarrow 0.
  \end{align*}
\end{MyDefinition}

\begin{MyLemma}\label{local_error_cond}
  Let $p$ be a positive integer. If $y$ is of piecewise class $C^{p}$ then
  the local discretization error function $hd^{[n]}$ satisfy
  \mathindent=0em
  \begin{align}\label{local_error}
    &hd^{\, [n] \; i}_{-c_ih}(\alpha h) =-\sum\limits_{k=1}^{p-1} \Gamma_{ik}(\alpha) y^{(k)}(t_{n-1})h^k +O(h^{p}),\quad \alpha\in[0,c_i],\quad i=1,\dots,s+1,
    \intertext{where $\Gamma_{ik}:\ [0,c_i]\rightarrow \mathbb{R}$ are polynomial functions given by} \notag
  \end{align}
  \vspace{-2.5\baselineskip}
  \begin{alignat}{2}\label{Gamma_i}
    \Gamma_{ik}(\alpha)=\frac{1}{(k-1)!}\biggl[\frac{(1-u_i(\alpha))(-1)^k}{k}  &+            \sum\limits_{j=1}^{s}{\widetilde{a}_{ij}}(\alpha)(-(1-c_j))^{k-1} & &+\sum\limits_{j=1}^{s} a_{ij}(\alpha)c_j^{k-1} -\frac{\alpha^k}{k} \biggr],\\
                                                                                &\hphantom{+} \alpha\in[0,c_i],                                                 & &\hphantom{+}i=1,\dots,s+1.\notag
  \end{alignat}
\end{MyLemma}

\begin{proof}
  Consider $hd^{[n]\;i}$ defined by~\eqref{TS_GLMs_matrix_form_residual:Z} for $\omega\in[-c_ih,0]$,
  make the change of the independent variable: $\omega=(\alpha -c_i) h$
  and use~\eqref{shifted_coeff}.
  The proof follows by Taylor series expansion about~$t_{n-1}$.
\end{proof}

\noindent For convenience we denote $\Gamma_{k}:=\Gamma_{s+1,\;k}$. Using~\eqref{coeff_notations} and~\eqref{Gamma_i} we obtain \mathindent=0em
\begin{alignat}{2}\label{Gamma}
  \Gamma_{k}(\alpha)=\frac{1}{(k-1)!}\biggl[\frac{(1-v(\alpha))(-1)^k}{k} &+\sum\limits_{j=1}^{s}{\widetilde{b}_{j}}(\alpha)(-(1-c_j))^{k-1} & &+\sum\limits_{j=1}^{s} b_{j}(\alpha)c_j^{k-1} -\frac{\alpha^k}{k} \biggr],\\
                                                                          &\hphantom{+}\alpha\in[0,h].  \notag
\end{alignat}
\begin{MyRemark}
  If the conditions~\eqref{one_step_conditions} hold the $\Gamma_{ik},\ \Gamma_{k}$

  are the same as for the one-step RK method~\cite{Maset_Zennaro,Maset}.
\end{MyRemark}
In the following we assume that the TSRK method satisfies the conditions $\Gamma_1=0$ and\linebreak[4] $\Gamma_{i1}=0,\; i=1,\dots,s,$ that is
\begin{alignat}{6}\label{order_1_cond}
  &v(\alpha)-1   & &+\sum\limits_{j=1}^{s}{\widetilde{b}_j}   (\alpha) & &+\sum\limits_{j=1}^{s} b_j(\alpha)    & &\ =\alpha, &\qquad &\alpha\in[0,1],\notag\\
  \\
  &u_i(\alpha)-1 & &+\sum\limits_{j=1}^{s}{\widetilde{a}_{ij}}(\alpha) & &+\sum\limits_{j=1}^{s} a_{ij}(\alpha) & &\ =\alpha, &       &\alpha\in[0,c_i],\quad i=1,\dots,s.\notag
\end{alignat}
The above condition is an equivalent form of uniform stage order \begin{emph_}one\end{emph_} condition.

The proofs of the Theorems~\ref{order_2_cond},~\ref{order_3_cond},~\ref{order_4_cond} are not difficult but rather technical.
We omit them for the sake of brevity.

\begin{MyTheorem}\label{order_2_cond}
  The TSRK method satisfying~\eqref{order_1_cond} has
  uniform order two iff $\ \Gamma_2=0$.
\end{MyTheorem}

\begin{MyTheorem}\label{order_3_cond}
  Let the TSRK method satisfy~\eqref{order_1_cond} and have uniform order two.

\noindent
  If $\ \Gamma_3=0\ $ and $\sum\limits_{\substack{i=1\\ c_i=c_m^*}}^s {b_i(\alpha) \Gamma_{i\, 2}(\beta)}=0,\quad \alpha\in[0,1],\quad \beta\in[0,c_m^*],\quad m=1,\dots,s^*$

\noindent
  then the method has uniform order three.
\end{MyTheorem}

\begin{MyTheorem}\label{order_4_cond}
  Let the TSRK method satisfy~\eqref{order_1_cond} and have uniform order three.
  \mathindent=0em
\begin{alignat}{5}
  \text{If }\ &\Gamma_4=0, \notag\\
              &\sum\limits_{\substack{i=1 \\ c_i=c_m^*}}^s {b_i(\alpha) \Gamma_{i\, 3}(\beta)}=0,                                                           &\quad \alpha&\in[0,1],&\quad &\beta\in[0,c_m^*], &\quad        &              &\quad   m&=1,\dots,s^*,\\
              &\sum\limits_{\substack{i=1 \\ c_i=c_m^*}}^s {\sum\limits_{\substack{j=1\\ c_j=c_l^*}}^s b_i(\alpha) a_{ij}(\beta) \Gamma_{j\, 2}(\gamma)}=0, &      \alpha&\in[0,1],&      &\beta\in[0,c_m^*], &\quad \gamma &\in[0,c_l^*], &\quad l,m&=1,\dots,s^*.\notag
\end{alignat}

\noindent
  then the method has uniform order four.
\end{MyTheorem}

\begin{MyTheorem}\label{stage_order}
  The TSRK method has uniform stage order~$\widetilde{q}\;$ iff\\
  $\Gamma_{ik}=0,\ \Gamma_{k}=0,\quad i=1,\dots,s,\quad k=1,\dots,\widetilde{q}.$
\end{MyTheorem}
\begin{proof}
  It follows from~\eqref{local_error}.
\end{proof}

The following results can be obtained as corollaries of Theorems~\ref{stage_order} and~\ref{order_3_cond},~\ref{order_4_cond}.

\begin{MyCorollary}\label{stage_order_3_cond}
  Let the TSRK method have uniform stage order two.

\noindent
  If $\ \Gamma_3=0\ $  then the method has uniform order three.
\end{MyCorollary}

\begin{MyCorollary}\label{stage_order_4_cond}
  Let the TSRK method have uniform stage order three.

\noindent
  If $\ \Gamma_4=0$ then the method has uniform order four.
\end{MyCorollary}

The results of Corollary~\ref{stage_order_3_cond} and~\ref{stage_order_4_cond} can be easily generalized as follows.

\begin{MyTheorem}\label{order_p_cond}
  Let the TSRK method have uniform stage order~$\widetilde{q}$.\linebreak[3]
 It has uniform order~$q=\widetilde{q}+1$ iff $\ \Gamma_{\widetilde{q}+1}=0$.
\end{MyTheorem}

\section{Construction of explicit TSRK methods of\\ uniform stage order four and five}

Consider a two-stage explicit TSRK method satisfying~\eqref{order_1_cond}. 
Its Butcher tableau is

\begin{MyTable}{Butcher tableau for 2-stage explicit TSRK methods}\label{tab_2_stage_ETSGLMs}
  \arrayrulewidth=0.5pt
  \begin{tabular}{c|c|c c|c c}
    $c_1$   & $u_1(\alpha)$ & $\widetilde{a}_{11}(\alpha)$  & $\widetilde{a}_{12}(\alpha)$  & 0                 & 0             \\
    $c_2$   & $u_2(\alpha)$ & $\widetilde{a}_{21}(\alpha)$  & $\widetilde{a}_{22}(\alpha)$  & $a_{21}(\alpha)$  & 0             \\

    \hline
          & $v(\alpha)$   & $\widetilde{b}_{1}(\alpha)$   & $\widetilde{b}_{2}(\alpha)$   & $b_1(\alpha)$     & $b_{2}(\alpha)$
  \end{tabular}
  \medskip
\end{MyTable}
A natural choice will be to space out  the abscissae $c_i,\; i=1,\dots,s$ uniformly
\smallskip
in the interval~$[0,1]$ so that~\cite{Butcher_Jackiewicz}
\smallskip
$c_1=0,\ c_2=\dfrac{1}{s-1},\dots,\ c_{s-1}=\dfrac{s-2}{s-1},\ c_s=1$. In the case of $s=2$ we have $c_1=0,\ c_2=1$.

Since $c_1=0,$ conditions $\Gamma_{1k}(\alpha)=0,\quad\alpha\in[0,c_1],\; k=1,2,\dots$
reduce to $\Gamma_{1k}(0)=0,$ \linebreak[4] $\ k=1,2,\dots$ that follows from~\eqref{Yi_continuous}. It also follows that $u_1(\cdot)=1,\ \widetilde{a}_{11}(\cdot)=0,\ \widetilde{a}_{12}(\cdot)=0$.

For the sake of brevity we omit the argument $\alpha$ of the method coefficient functions. By~\hbox{Theorem}~\ref{order_p_cond}, the method has uniform order four and uniform stage order three if
\linebreak[4]
 $\Gamma_k=0,\; k=1,2,3,4$ and $\Gamma_{2\,k}=0,\; k=1,2,3,\;$ that is
\begin{alignat}{5}\label{2_stage_4_order_ETSGLM_cond}
  -(1-v)           &+\widetilde{b}_1     & &+\widetilde{b}_2     & &+b_1    & &+b_2      & &=\alpha,            \notag\\
  \frac{1-v}{2}    &-\widetilde{b}_1     & &                     & &        & &+b_2      & &=\frac{\alpha^2}{2},\notag\\
  -\frac{1-v}{3}   &+\widetilde{b}_1     & &                     & &        & &+b_2      & &=\frac{\alpha^3}{3},\notag\\
   \frac{1-v}{4}   &-\widetilde{b}_1     & &                     & &        & &+b_2      & &=\frac{\alpha^4}{4},\\
  -(1-u_2)         &+\widetilde{a}_{21}  & &+\widetilde{a}_{22}  & &+a_{21} & &          & &=\alpha,            \notag\\
  \frac{1-u_2}{2}  &-\widetilde{a}_{21}  & &                     & &        & &          & &=\frac{\alpha^2}{2},\notag\\
  -\frac{1-u_2}{3} &+\widetilde{a}_{21}  & &                     & &        & &          & &=\frac{\alpha^3}{3},\notag
\end{alignat}
where $\alpha\in[0,1]$.

\noindent The coefficients are defined by
\begin{alignat}{5}\label{2_stage_4_order_ETSGLM_coeff}
u_{{2}}               &=- \left( 2\,\alpha-1 \right)  \left( \alpha+1 \right) ^{2},\notag\\
v                     &= \left( \alpha-1 \right) ^{2} \left( \alpha+1 \right) ^{2},\notag\\
\widetilde{a}_{2,1}   &={\alpha}^{2} \left( \alpha+1 \right),\notag\\
\widetilde{b}_{1}     &=-\frac{1}{12}\,{\alpha}^{2} \left( \alpha+1 \right)  \left( 5\,\alpha-7 \right),\\
a_{2,1}               &= \alpha\, \left( \alpha+1 \right)^2-\widetilde{a}_{2,2},\notag\\
b_1                   &=-\frac{1}{3}\,\alpha\, \left( 2\,\alpha-3 \right)  \left( \alpha+1 \right) ^{2}-\widetilde{b}_{2},\notag\\
b_2                   &=\frac{1}{12}\,{\alpha}^{2} \left( \alpha+1 \right) ^{2},\notag
\end{alignat}
where $\widetilde{a}_{2,2},\ \widetilde{b}_{2}$ remain free. The relation $\Gamma_5(1)=\dfrac{4}{15}\ne0$ implies that it is impossible to attain discrete order five.

The uniform order and the uniform stage order can be increased by finding a suitable value for $c_2$. Assume that $c_1=0,\ c_2  \ne 0$ (in general case~$c_2 \ne 1$). By theorem~\eqref{order_p_cond},
the method has uniform order five and uniform stage order four if\linebreak[3] $\Gamma_k=0,\; k=1,2,3,4,5$ and $\Gamma_{2\,k}=0,\; k=1,2,3,4, \;$ that is

\begin{alignat}{5}\label{2_stage_5_order_ETSGLM_cond}
  -(1-v)           &+\widetilde{b}_1     & &+\widetilde{b}_2              & &+b_1    & &+b_2       & &=\alpha,            \notag\\
  \frac{1-v}{2}    &-\widetilde{b}_1     & &-(1-c_2)\widetilde{b}_2       & &        & &+c_2   b_2 & &=\frac{\alpha^2}{2},\notag\\
  -\frac{1-v}{3}   &+\widetilde{b}_1     & &+(1-c_2)^2\widetilde{b}_2     & &        & &+c_2^2 b_2 & &=\frac{\alpha^3}{3},\notag\\
  \frac{1-v}{4}    &-\widetilde{b}_1     & &-(1-c_2)^3\widetilde{b}_2     & &        & &+c_2^3 b_2 & &=\frac{\alpha^4}{4},\notag\\
  -\frac{1-v}{5}   &+\widetilde{b}_1     & &+(1-c_2)^4\widetilde{b}_2     & &        & &+c_2^4 b_2 & &=\frac{\alpha^5}{5},\\
  -(1-u_2)         &+\widetilde{a}_{21}  & &+\widetilde{a}_{22}           & &+a_{21} & &           & &=\alpha,            \notag\\
  \frac{1-u_2}{2}  &-\widetilde{a}_{21}  & &-(1-c_2)\widetilde{a}_{22}    & &        & &           & &=\frac{\alpha^2}{2},\notag\\
  -\frac{1-u_2}{3} &+\widetilde{a}_{21}  & &+(1-c_2)^2\widetilde{a}_{22}  & &        & &           & &=\frac{\alpha^3}{3},\notag\\
  \frac{1-u_2}{4}  &-\widetilde{a}_{21}  & &-(1-c_2)^3\widetilde{a}_{22}  & &        & &           & &=\frac{\alpha^4}{4},\notag
\end{alignat}
where $\alpha\in[0,1]$ in the first five equations~\eqref{2_stage_5_order_ETSGLM_cond} and $\alpha\in[0,c_2]$ in other ones.

\noindent The coefficients are defined by \mathindent=0em \medmuskip=3mu
\begin{alignat}{5}\label{2_stage_5_order_ETSGLM_coeff}
               u_2&=\left( \alpha+1 \right) ^{2} \left( 1-2\,\alpha+{\frac {3{\alpha}^{2}}{2\,c_{{2}}-1}} \right),  \notag\\
                 v&=-{\frac { \left( \alpha+1 \right) ^{2} \left(  \left( 10\,\alpha-5 \right) {c_{{2}}}^{2}-15\,c_{{2}}{\alpha}^{2}+ \left( \alpha+1 \right)  \left( 6\,{\alpha}^{2}-3\,\alpha+1 \right)  \right) }{5\,{c_{{2}}}^{2}-1}},  \notag\\
\widetilde{a}_{21}&={\alpha}^{2} \left( \alpha+1 \right) -{\frac {{\alpha}^{2}\left( \alpha+1 \right) ^{2} \left( 3\,c_{{2}}-1 \right) }{2\,c_{{2}}\left( 2\,c_{{2}}-1 \right) }},  \notag\\
\widetilde{a}_{22}&={\frac {{\alpha}^{2} \left( \alpha+1 \right) ^{2}}{2\,c_{{2}} \left( c_{{2}}-1 \right)  \left( 2\,c_{{2}}-1 \right) }},  \notag\\
   \widetilde{b}_1&={\frac {{\alpha}^{2} \left( \alpha+1 \right)  \left( 20{c_{{2}}}^{4}\!-\left(30 \alpha+10 \right) {c_{{2}}}^{3}+ \left( 12{\alpha}^{2}\!+\!3\alpha\!-\!13 \right) {c_{{2}}}^{2}+ \left( 4{\alpha}^{2}\!+\!11\alpha\!+\!3 \right) c_{{2}}-2\alpha\left( \alpha+1 \right)  \right)}{ 4c_2\left( 5{c_2}^{2}-1\right)  \left( c_2+1 \right) }},  \notag\\
   \widetilde{b}_2&={\frac { {\alpha}^{2}\left( \alpha+1 \right)^{2}\left( 5\,{c_{{2}}}^{2}- \left( 4\,\alpha-3 \right) c_{{2}}-2\,\alpha \right) }{4 c_2 \left( 5\,{c_{{2}}}^{2}-1 \right)\left( c_2-1 \right) }}\\
            a_{21}&=\alpha \left( \alpha+1 \right) ^{2} \left( 1-{\frac {\alpha\, \left( 3\,c_{{2}}-2 \right) }{2 \left( 2\,c_{{2}}-1 \right)  \left( c_{{2}}-1 \right) }} \right),  \notag\\
               b_1&={\frac {\alpha\left( \alpha+1 \right)^2  \left( 20{c_2}^{4}\!-\left(30 \alpha+20 \right) {c_{{2}}}^{3}+ \left( 12{\alpha}^{2}\!+\!21\alpha\!-\!4 \right) {c_{{2}}}^{2}+ \left( -4{\alpha}^{2}\!+\!3\alpha\!+\!4 \right) c_{{2}}-2\alpha\left( \alpha+1 \right)  \right)}{ 4c_2\left( 5{c_2}^{2}-1\right)  \left( c_2-1 \right) }},  \notag\\
               b_2&=-{\frac { {\alpha}^{2}\left( \alpha+1 \right)^{2}\left( 5\,{c_{{2}}}^{2}- \left( 4\,\alpha+7 \right) c_{{2}}+2\,\alpha+2 \right) }{4 c_2 \left( 5\,{c_{{2}}}^{2}-1 \right)\left( c_2+1 \right) }}.  \notag
\end{alignat}
To attain the discrete stage order five, we determine $c_2$ from $\Gamma_{2\,5}(1)=0$. We have \linebreak[4]
\vspace{-\baselineskip}
\begin{equation}
  \ c_2=\dfrac{11-\sqrt{41}}{10}.
\end{equation}
\mathindent=2.5em
\vspace{-\baselineskip}

\noindent The relation $\Gamma_6(1)=-{\dfrac {16\ \ \bigl(17-2\sqrt {41}\,\bigr)}{75\,\bigl(71-11\sqrt {41}\,\bigr)}\ne 0}$ implies that it is impossible to attain discrete order six.\linebreak[3]

\vspace{-6pt}
So we construct the explicit TSRK method of uniform order five, uniform stage order four and discrete stage order five.
\enlargethispage{\baselineskip}
\vspace{-0.5\baselineskip}
\begin{MyRemark}
  There is not a method of uniform stage order two in a class of explicit one-step RK methods for RFDEs.

  Indeed, for explicit one-step RK methods $c_1=0,\ c_2\ne0$ and  $a_{2,j}=0,\ j=2,\dots,s$,
  hence $\Gamma_{2\,k}=-\dfrac{\alpha^k}{k!}\ne0,\quad \alpha\in(0,c_2], \ k=2,3,\dots\ \ .$

\vspace{-6pt}

  It is known~\cite{Dekker} that methods with low stage order suffer from the order reduction phenomenon when applied to stiff ODEs.
  Hence, the explicit TSRK methods may be more appropriate for some mildly stiff RFDEs.
\end{MyRemark}
\vspace{-0.5\baselineskip}

\end{document}